\documentclass[a4paper, 12pt]{amsart}
\usepackage{amscd}
\usepackage{amsmath}
\usepackage{amssymb}
\usepackage{amsthm}
\usepackage{tikz}
\usetikzlibrary{snakes}
\usepackage{verbatim}
\usepackage[all]{xy}
\usepackage{color}
\usepackage[pagebackref=true]{hyperref}

\newtheorem{theorem}{Theorem}[section]
\newtheorem{lemma}[theorem]{Lemma}

\theoremstyle{plain}

\theoremstyle{definition}

\theoremstyle{definition} 
\newtheorem*{definition*}{Definition}

\theoremstyle{theorem} 
\newtheorem{conjecture}[theorem]{Conjecture}

\theoremstyle{theorem} 
\newtheorem*{conjecture*}{Conjecture}

\theoremstyle{theorem}

\theoremstyle{definition} 
\newtheorem{assumption}[theorem]{Assumption}

\theoremstyle{definition} 
\newtheorem*{assumption*}{Assumption}

\theoremstyle{remark}

\theoremstyle{remark} 
\newtheorem*{remark*}{Remark}

\newcommand{\one}{\mathbf{1}}

\newcommand{\Lii}{L^{2}}
\newcommand{\Lp}{L^{p}}
\newcommand{\Lq}{L^{q}}

\newcommand{\Linf}{L^{\infty}}

\providecommand{\floor}[1]{\lfloor#1\rfloor}

\providecommand{\abs}[1]{\lvert#1\rvert}
\providecommand{\Abs}[1]{\left|#1\right|}
\providecommand{\norm}[1]{\lVert#1\rVert}
\providecommand{\Norm}[1]{\left\|\,#1\,\right\|}

\def\b{{\beta}}
\def\d{{\delta}}
\def\e{{\epsilon}}

\def\t{{\tau}}

\title[Limit theorems]{Limit theorems for rank-one Lie groups}
\author[A. Gorodnik]{Alexander Gorodnik}
\author[F. A. Ram{\'i}rez]{Felipe A. Ram{\'i}rez}

\date{\today}

\date{}
\address{School of Mathematics, University of Bristol, Bristol, UK}
\email{a.gorodnik@bristol.ac.uk}
\address{School of Mathematics, University of Bristol, Bristol, UK}
\email{f.a.ramirez@bristol.ac.uk}

\begin{document}


\begin{abstract}
We investigate asymptotic behaviour of averaging operators for actions of  simple rank-one Lie groups.
It was previously known that these averaging operators converge almost everywhere, and we establish
a more precise asymptotic formula that describes their deviations from the limit.
\end{abstract}


\maketitle 


\section{Introduction}

Let $(X,\mu)$ be a probability space and $\{T_t\}_{t\in\mathbb{R}}$ a measure preserving 
one-parameter flow on $X$. It is a fundamental problem in ergodic theory to analyse 
the statistical properties of ``observables'' $f(T_tx)$ 
defined for a measurable function $f:X\to \mathbb{C}$ and a point $x\in X$.
In particular, one would like to understand asymptotic behaviour of the time averages
$$
(\mathcal{A}_t f)(x):=\frac{1}{t}\int_0^t f(T_sx)\,ds.
$$
It is natural to expect that for chaotic flows the quantities $f(T_tx)$
should behave similarly to independent identically distributed random variables.
Indeed, assuming that the flow is ergodic, it follows from the Pointwise Ergodic Theorem
that for every $f\in L^1(X)$, the average $\mathcal{A}_tf$ converges almost everywhere to
$\int_X f\, d\mu$ as $t\to\infty$,
which is reminiscent of the Law of Large Numbers, and assuming that the flow is aperiodic,
there exists a function $f\in L^2(X)$ such that
the normalised deviations 
$\frac{\mathcal{A}_tf - \int_X f\,d\mu}{\|\mathcal{A}_t f- \int_X f\,d\mu\|_2}$ converge in distribution to
the standard normal law as $t\to\infty$ (see \cite{BD87}), which is reminiscent of the Central Limit Theorem.
The convergence of deviations to the normal law is a widespread phenomenon
that holds, for instance, for all sufficiently smooth functions in the setting of the geodesic flows
on compact manifolds of negative curvature (see \cite{s,r}).
  
In this paper we are interested in
the behaviour of deviations of averages for actions of an interesting
class of large groups---the connected simple rank-one Lie groups with finite centre.
For instance, our results apply to the groups of orientation preserving isometries
of the hyperbolic spaces. 

Let $G$ be a connected simple rank-one Lie group with finite centre
that acts measure-preservingly on a standard probability space $(X,\mu)$. We fix a maximal compact subgroup $K$
in $G$ and consider the sets
\begin{equation}\label{eq:b_t}
B_t:=\{g\in G:\, d(gK,K)\le t\},
\end{equation}
where $d$ is the canonical Riemannian metric on the corresponding symmetric space $G/K$. 
Given a measurable function $f:X\to \mathbb{C}$ and point $x\in X$, we define the averaging operators
\begin{equation}\label{eq:a}
(\mathcal{A}_t f)(x):=\frac{1}{m(B_t)}\int_{B_t} f(g^{-1}x)\,dm(g),
\end{equation}
where $m$ denotes a Haar measure on $G$. A Pointwise Ergodic Theorem for 
these operators has been established in \cite{Nev94,Nev97,NS97}. It shows that
for every $f\in L^2(X)$,
\begin{equation}\label{eq:m}
\mathcal{A}_t f\to P_0f\quad\hbox{as $t\to\infty$, }
\end{equation}
in $L^2$-norm and almost everywhere, 
where $P_0$ denotes the orthogonal projection on the subspace of $G$-invariant functions.
Our aim here is to investigate asymptotics of the deviations $\mathcal{A}_tf - P_0 f$.
It turns out that they exhibit a very different behaviour compared with one-parameter flows.
For instance, $\frac{\mathcal{A}_tf - P_0 f}{\|\mathcal{A}_t f- P_0 f\|_2}$
might converge almost everywhere (see Theorem \ref{discretetime} below).

Our arguments are based on representation-theoretic techniques.
We denote by $\hat G$ the set of equivalence classes of irreducible unitary representations of $G$
equipped with the Fell topology.
The action of $G$ on $X$ defines a natural unitary representation of $G$ on $L^2(X)$ which can be
disintegrated as
\begin{equation}\label{eq:decomp}
L^2(X)=\int_{\hat G}^\oplus d_\sigma\mathcal{H}_\sigma\, d\nu(\sigma),
\end{equation}
where $\sigma:G\to U(\mathcal{H}_\sigma)$ are the corresponding irreducible representations,
$d_\sigma\in \mathbb{N}\cup\{\infty\}$ their multiplicities, and  
$\nu$ is a Borel measure on $\hat G$.
Let $\hbox{Spec}_G(X)\subset \hat G$ be the support of the measure $\nu$.
We denote by $\hat G^{1}$ the subset of $\hat G$ consisting of spherical representations, that is,
representations $\sigma:G\to U(\mathcal{H}_\sigma)$ such that $\mathcal{H}_\sigma$ contains a nonzero
$K$-invariant vector.
Since the averaging operators $\mathcal{A}_t$ are bi-invariant under $K$,
only the spherical representations will play an essential role in our analysis.
The set $\hat G^1$ has been described for rank-one groups by Kostant \cite{k},
and it is customary to parametrise it as
\begin{equation}\label{eq:dual}
\hat G^1\simeq \{\rho\}\cup (0,\rho']\cup i\mathbb{R}^+,
\end{equation}
where $0<\rho'\le \rho$. Here $\{\rho\}$ corresponds to the trivial representation,
$(0,\rho']$ corresponds to the complementary series representations, and 
$i\mathbb{R}^+$ corresponds to the principle series representations.
For $G=\hbox{SO}^\circ(n,1)$, we have $\rho=\rho'=\frac{n-1}{2}$.

A rich collection of examples of spaces equipped with actions of $G$, which play an important role in various
number-theoretic applications, is provided by homogeneous spaces of algebraic groups.
Let $H\subset \hbox{GL}_d(\mathbb{R})$ be a semisimple real algebraic group defined over $\mathbb{Q}$.
Its congruence subgroup 
$$
\Gamma:=\{\gamma \in H\cap \hbox{GL}_d(\mathbb{Z}):\, \gamma=\hbox{id}\mod q\}
$$
has finite covolume in $H$, and we take $X=H/\Gamma$
equipped with the normalised invariant measure $\mu$.
Let us also assume that $G$ is the connected component of a real algebraic subgroup of $H$ defined over
$\mathbb{Q}$. Then we obtain a natural measure-preserving action $G \curvearrowright (X,\mu)$.
We will call such actions {\it arithmetic}.
Description of $\hbox{Spec}_G(X)$ for arithmetic actions 
is the central problem in the theory of automorphic representations (see, for instance, \cite{sa}). 

The following conjecture has been formulated in \cite{BLS92}:

\begin{conjecture}[Purity]\label{c:purity}
The spectrum of arithmetic actions $G \curvearrowright (X,\mu)$ for $G=\hbox{\rm SO}^\circ(n,1)$ satisfies
$$
\hbox{\rm Spec}_G(X)\cap \hat G^1\subset \bigcup_{0\le j<\frac{n-1}{2}}\left\{\frac{n-1}{2}-j\right\}\cup i\mathbb{R}^+.
$$
\end{conjecture}

Although this conjecture was formulated only for homogeneous spaces of $\hbox{\rm SO}^\circ(n,1)$,
it follows from \cite[Th.~1.1(a)]{BS91} that it must also hold for general arithmetic actions
of $\hbox{\rm SO}^\circ(n,1)$.

It was demonstrated in \cite{BLS92} that there are arithmetic actions for which 
representations corresponding to $\frac{n-1}{2}-j$ do occur in $\hbox{Spec}_G(X)$.
Conjecture \ref{c:purity} was partially proved by Bergeron and Clozel \cite{BC10} who showed that
for arithmetic actions of $\hbox{\rm SO}^\circ(n,1)$,
$$
\hbox{\rm Spec}_G(X)\cap \hat G^1\subset \bigcup_{0\le j<\frac{n-1}{2}}\left\{\frac{n-1}{2}-j\right\}
\cup \left(0,\frac{1}{2}-\frac{1}{N^2+1}\right]\cup i\mathbb{R}^+,
$$
where $N=n$ if $n$ is even and $N=n+1$ if $n$ is odd.
We note that a similar purity phenomenon is expected to hold for other rank-one groups.
In particular, it was shown in \cite[Ch.~6]{bc} that Arthur's conjectures \cite{a} imply that 
such a property holds for arithmetic actions of $\hbox{SU}(n,1)$.

\vspace{0.3cm}

Motivated by Conjecture \ref{c:purity} and supporting results from \cite{BC10},
we investigate asymptotic behaviour of the averaging operators $\mathcal{A}_t$
under the following purity assumption (see Figure \ref{p:spec}):

\begin{assumption}[Purity]\label{a:purity}
A measure-preserving action of $G$ on $(X,\mu)$ satisfies
$$
\hbox{\rm Spec}_G(X)\cap \hat G^1=\{s_0,\ldots,s_k\}\cup\Omega
$$
for $s_0=\rho > s_1 >\dots> s_k > r$ and a subset $\Omega$ of $(0,r]\cup i\mathbb{R}^+$.
\end{assumption}

\begin{figure}[h]\label{p:spec}
\begin{tikzpicture}[>=latex]
	\draw[thin] (0,0)--(5.5,0);
	\draw[thin, densely dotted] (5.5,0)--(7,0);
	\draw[line width=2.5pt] (0,0)--(1.5,0);
	\draw[->,line width=2.5pt, cap=round] (0,0)--(0,6);
	\fill (2,0) circle (2pt) (3,0) circle (2pt) (4,0) circle (2pt) (5,0) circle (2pt) (5.5,0) circle (1pt) (7,0) circle (2pt);
	\draw (7,2) node(rho) [draw] {$s_0=\rho$};
	\draw[->, shorten >=2pt] (rho)--(7,0);
	\draw (5.5,3) node(s0) [draw] {$\rho'$};
	\draw[->, shorten >=2pt, densely dotted] (s0)--(5.5,0);
	\draw (5,2) node(s1) [draw] {$s_1$};
	\draw[->, shorten >=2pt] (s1)--(5,0);
	\draw (4,2) node(s2) [draw] {$s_2$};
	\draw[->, shorten >=2pt] (s2)--(4,0);
	\draw (3,2) node(sk1) [draw] {$s_{k-1}$};
	\draw[->, shorten >=2pt] (sk1)--(3,0);
	\draw (2,2) node(sk) [draw] {$s_k$};
	\draw[->, shorten >=2pt] (sk)--(2,0);
	\draw (1.5,1) node(z0) [draw] {$r$};
	\draw[->, shorten >=5pt] (z0)--(1.5,0);
	\draw (3.5,0.25) node(dots) {\dots};
	\draw (3.5,4) node(suppnu) [draw] {$\Omega$};
	\draw[->, shorten >=5pt] (suppnu)--(0,4);
\end{tikzpicture}
\caption{Purity.}
\label{illustration}
\end{figure}

Our first result concerns convergence in $L^2$-norm.

\begin{theorem}\label{mean}
Let  $G \curvearrowright (X,\mu)$ be a measure-preserving action 
on a standard probability space satisfying Assumption~\ref{a:purity}.  
Then there exist pairwise orthogonal projections $P_0,\ldots,P_k$ on $\Lii(X)$ and smooth functions 
$\psi_0,\ldots,\psi_k$ such that
$$
\psi_0=1,\quad \quad\psi_j (t) \sim_{t\to\infty} c_j\, e^{-(\rho-s_j) t}\;\;\hbox{with $c_j>0$},\;\;
j=1,\ldots, k,
$$
and for every $f \in \Lii(X)$ and $t\ge 1$,
\[
	\Norm{\mathcal{A}_tf - \sum_{j=0}^{k}\psi_j (t)\, P_jf}_{2} \ll\,t e^{-(\rho-r) t}\,\norm{f}_2.
\]
\end{theorem}

The projection $P_0$ in Theorem \ref{mean} is the orthogonal projection on the subspace
of $G$-invariant functions in $L^2(X)$. When $k=0$, Theorem \ref{mean} 
gives a quantitative Mean Ergodic Theorem for the operators $\mathcal{A}_t$ as in (\ref{eq:m}).

\vspace{0.3cm}

We also establish an almost everywhere asymptotic formula for the averaging operators at
integral times.

\begin{theorem}\label{discretetime}
Let  $G \curvearrowright (X,\mu)$ be a measure-preserving action 
on a standard probability space satisfying Assumption~\ref{a:purity}.  
Then there exist pairwise orthogonal projections $P_0,\ldots,P_k$ on $\Lii(X)$ and functions 
$\psi_0,\ldots,\psi_k$ as in Theorem \ref{mean} such that for every $f \in \Lii(X)$,
almost every $x\in X$, and $n\in\mathbb{N}$,
\[
	\Abs{(\mathcal{A}_nf)(x) - \sum_{j=0}^{k}\psi_j (n)\,(P_j f)(x)} \leq
        C_{\e}(x,f)\,n^{3/2+\e}e^{-(\rho-r)n },\quad \e>0,
\]
where $\norm{C_\e(\cdot,f)}_2 \ll_\e\,\norm{f}_2$.
\end{theorem}

Theorem \ref{discretetime}, in particular, implies that for $f\in L^2(X)$ with $P_1f\ne 0$,
$$
\|\mathcal{A}_nf -P_0f\|_2\approx e^{-(\rho-s_1)t}, 
$$
and $\frac{\mathcal{A}_nf -P_0f}{\|\mathcal{A}_nf -P_0f\|_2}$ converges almost everywhere
to $\frac{P_1f}{\|P_1f\|_2}$ as $n\to \infty$.

\vspace{0.3cm}

Using an interpolation argument, we deduce an
almost everywhere asymptotic formula for the averaging operators at continuous times.

\begin{theorem}\label{pointwise}
Let  $G \curvearrowright (X,\mu)$ be a measure-preserving action 
on a standard probability space satisfying Assumption~\ref{a:purity}.  
Then there exist pairwise orthogonal projections $P_0,\ldots,P_k$ on $\Lii(X)$ and functions 
$\psi_0,\ldots,\psi_k$ as in Theorem \ref{mean} such that 
for every $f \in \Lp(X)$, $p>2$, almost every $x\in X$, and $t\ge 1$,
\[
	\Abs{(\mathcal{A}_tf)(x) - \sum_{j=0}^{k}\psi_j (t)\,(P_j f)(x)} \leq C_{p}(x,f)\,e^{-(1/2-1/p)(\rho-r)t},
\]
where $\norm{C_{p}(\cdot,f)}_2 \ll_{p}\,\norm{f}_p$.
\end{theorem}

We note that the interpolation argument in Theorem \ref{pointwise} inevitably diminishes
the quality of the error term, and only the summands with 
$$
s_j>(1/2-1/p)r+(1/2+1/p)\rho
$$
provide nontrivial information because the other summands are of smaller order than the error term.

\subsection*{Acknowledgements}
The first author is support by EPSRC, ERC and RCUK, and the second author is
supported by ERC. We would like to thank A.~Nevo for a useful discussion.


\section{Preliminaries}

Throughout the paper, $G$ is a connected simple rank-one Lie group with finite centre,
and $K$ is a maximal compact subgroup of $G$. We fix a one-parameter subgroup 
$A=\{a_t\}_{t\in \mathbb{R}}$ such that the Cartan decomposition
$$
G=K\,\{a_t\}_{t\ge 0}\,K
$$
holds and $d(a_t K, K)=t$,
where $d$ denotes the canonical Riemannian metric on the symmetric space $G/K$.

Let $m$ be a Haar measure on $G$. With respect to the Cartan decomposition,
it can be given as \cite[Ch.~10]{h}
$$
dm(k_1,t,k_2)=dm_K(k_1)\,\Delta(t)dt\,dm_K(k_2),
$$
with the probability Haar measure $m_K$ on $K$ and 
$$
\Delta(t)=(\sinh t)^{n_1} (\sinh 2t)^{2n_1},
$$
where $n_1,n_2$ denote the dimensions of the corresponding root spaces.
In particular, for the Riemannian ball $B_t$, defined in (\ref{eq:b_t}), we have 
\begin{equation}\label{eq:vol}
m(B_t)=\int_0^t \Delta(t)\,dt\sim_{t\to\infty} c\, e^{2\rho t}
\end{equation}
with $c>0$ and $\rho=(n_1+2n_2)/2$. Note that the parameter $\rho$ here is the same as in
(\ref{eq:dual}). It also follows from (\ref{eq:vol}) that
for every $t\ge 1$ and $\e\in (0,1)$,
\begin{equation}\label{eq:reg}
m(B_{t+\e}\backslash B_t)\ll \e\, m(B_t).
\end{equation}

We denote by $\beta_t$ the uniform probability measure on $G$ supported on the set $B_t$, namely,
$$
d\b_t(g) := \frac{\one_{B_t}(g)}{m (B_t)}\,dm(g).
$$
Given a unitary representation $\sigma:G\to U(\mathcal{H})$,
we define the corresponding averaging operator
$$
\sigma(\b_t):\mathcal{H}\to\mathcal{H}:v\mapsto \int_G \sigma(g)v\, d\b_t(g).
$$
In particular, for the unitary representation
$$
\pi(g)f(x)=f(g^{-1}x),\quad f\in L^2(X),
$$
induced by a measure-preserving action of $G$ on $(X,\mu)$, 
$\pi(\b_t)$ is equal to the operator $\mathcal{A}_t$ defined in (\ref{eq:a}).

If the representation $\sigma$ is irreducible, the subspace $\mathcal{H}^K$ of $K$-invariant vectors has
dimension at most one. If $\mathcal{H}^K$ has dimension one, $\sigma$ is called spherical.
In this case, we fix $v_\sigma\in \mathcal{H}^K$ with $\|v_\sigma\|=1$.
The corresponding spherical function is defined by
$$
\phi_\sigma(g):=\left<\sigma(g)v_\sigma,v_\sigma\right>.
$$
Using the identification (\ref{eq:dual}), we also write $\phi_s$ with $s\in \{\rho\}\cup (0,\rho']\cup
i\mathbb{R}^+$. An explicit integral formula for the spherical functions has been derived in
\cite{HC58a,HC58b}, and in the case of rank-one groups they can be also expressed in terms of Jacobi functions
\cite{Koo84}. We refer to the monograph \cite{GV88} for a comprehensive theory of spherical functions. 
In particular, we recall that by \cite[5.1.18]{GV88}, 
\begin{equation}\label{eq:01}
|\phi_s(a_t)|\ll e^{-(\rho-\hbox{\tiny Re}(s))t}(1+t),\quad t\ge 0,
\end{equation}
and by \cite[(2.19)]{Koo84}, when $0<s<\rho$
\begin{equation}\label{eq:02}
|\phi_s(a_t)|\sim_{t\to\infty} c(s)\,e^{-(\rho-s)t}
\end{equation}
for some $c(s)>0$.

\section{Proofs}

\begin{proof}[Proof of Theorem~\ref{mean}]
We first investigate the behaviour of averages $\sigma(\b_t)$ for irreducible representations $\sigma:G\to U(\mathcal{H})$.
Since the measure $\b_t$ is left $K$-invariant, it follows that 
$\sigma(\b_t)v$ is $K$-invariant for every $v\in\mathcal{H}$. 
In particular, $\sigma(\b_t)=0$ if $\sigma$ is not spherical.

Now we assume that $\sigma$ is spherical. Since $\b_t$ is right $K$-invariant,
for every $v\in\mathcal{H}$,
$$
\sigma(\b_t)v=\sigma(\b_t)\bar v,
$$
where
$$
\bar{v} := \int_{K} \sigma(k) v \,dm_K (k).
$$
Since $\dim (\mathcal{H}^K)=1$,
$$
\bar v=\left< \bar v, v_\sigma\right>v_\sigma=
\left(\int_{K} \left<\sigma(k) v,v_\sigma\right> \,dm_K (k)\right)v_\sigma
=\left<v, v_\sigma\right>v_\sigma,
$$
and
$$
\sigma(\b_t)v_\sigma=\left<\sigma(\b_t)v_\sigma,v_\sigma\right>v_\sigma=\psi_\sigma(t) v_\sigma,
$$
where
$$
\psi_\sigma(t):=\frac{1}{m (B_t)}\int_{B_t} \phi_\sigma(g)\,dm(g).
$$
Hence, we conclude that for any irreducible representation $\sigma$,
\begin{equation}\label{eq:irre}
\sigma(\b_t)=\psi_\sigma(t)\, P_\sigma,
\end{equation}
where $P_\sigma$ is the orthogonal projection on $\mathcal{H}^K$.
Clearly, the same formula also holds for representations which are direct sums of 
the representation $\sigma$.

Using the decomposition \eqref{eq:decomp}, every $f \in \Lii(X)$
can be written as
\[
	f = \int_{\hat G}^{\oplus}f_\sigma\,d\nu(\sigma),
\]
where $f_\sigma\in d_\sigma\mathcal{H}_\sigma$, and by \eqref{eq:irre},
$$
\pi(\beta_t)f = \int_{\hat G}^{\oplus}\sigma(\beta_t)f_\sigma\,d\nu(\sigma)
= \int_{\hat G^1}^{\oplus}\psi_\sigma(t)\, (P_{d_\sigma\sigma} f_\sigma)\,d\nu(\sigma).
$$
Taking into account Assumption \ref{a:purity}, we obtain
$$
\pi(\beta_t)f =\sum_{j=0}^k  \psi_j(t)\, P_j f+
\int_{\Omega}^{\oplus}\psi_\sigma(t)\, (P_{d_\sigma\sigma} f_\sigma)\,d\nu(\sigma),
$$
where $$
\psi_j(t):=\frac{1}{m (B_t)}\int_{B_t} \phi_{s_j}(g)\,dm(g),
$$
and $P_j$ are pairwise orthogonal projections.

Since by (\ref{eq:01}), 
$$
|\phi_\sigma(a_\t)|\ll e^{-(\rho-r)\t}(1+\t)\quad
\hbox{for all $\sigma\in \Omega$ and $\t\ge 0$,}
$$
we deduce that for all $\sigma\in\Omega$ and $t\ge 1$,
\begin{align*}
\abs{\psi_\sigma (t)} &\ll e^{-2\rho t}\int_{0}^{t}e^{2\rho \t}|\phi_\sigma(a_\t)|\,d\t\\
&=e^{-2\rho t}\int_{0}^{t}e^{(\rho+r)\t}(1+\tau)\,d\t \ll te^{-(\rho-r)t}.
\end{align*}
This implies that 
\begin{align*}
\left\|\pi(\beta_t)f -\sum_{j=0}^k  \psi_j(t)\, P_j f\right\|_2^2&=
\int_{\Omega}^{\oplus}|\psi_\sigma(t)|^2\, \|P_{d_\sigma\sigma} f_\sigma\|^2\,d\nu(\sigma)\\
&\ll \left(te^{-(\rho-r)t}\right)^2\|f\|_2^2,
\end{align*}
which is the main estimate of the theorem.
 
Finally, we observe that since $\phi_\rho=1$, we have $\psi_0=1$, and
it follows from l'H{\^o}pital's rule and (\ref{eq:02}) that
$$
\psi_j(t)=\frac{\int_0^t \phi_{s_j}(a_\tau)\Delta(\tau)\, d\tau}{\int_0^t \Delta(\tau)\, d\tau}
\sim_{t\to\infty} c_j\, e^{-(\rho-s_j)t}
$$
for some $c_j>0$. This completes the proof of the theorem. 
\end{proof}
 

\begin{proof}[Proof of Theorem~\ref{discretetime}]
Let $\e>0$.  For $f\in\Lii(X)$, we set
\[
	C_{\e}(\cdot,f):= \left(\sum_{n = 1}^{\infty} n^{-3-2\e} e^{2(\rho-r)n} \Abs{\pi(\b_{n})f - \sum_{j=0}^{k}\psi_j (n)\, P_j f}^2\right)^{1/2}.
\]
Then by Theorem~\ref{mean},
\begin{align*}
\norm{C_\e(\cdot,f)}_2^2 &\leq \sum_{n = 1}^{\infty} n^{-3-2\e}
e^{2(\rho-r)n}\Norm{\pi(\b_{n})f - \sum_{j=0}^{k}\psi_j (n)\, P_jf}_2^2\\
&\ll  \left(\sum_{n = 1}^{\infty} n^{-1-2\e}\right) \norm{f}_2^2 < \infty.
\end{align*}
This shows that $\Norm{C_\e (\cdot,f)}_2 \ll_\e\,\norm{f}_2$.
In particular, $C_\e(\cdot,f)$ is finite almost everywhere.  For almost every $x\in X$, 
\begin{align*}
	\Abs{\pi(\b_{n})f(x) - \sum_{j=0}^{k}\psi_j (n)\, (P_j f)(x)} \leq C_\e(x,f)\,n^{3/2+\e}e^{-(\rho-r) n},
\end{align*}
as desired.
\end{proof}


We start the proof of Theorem~\ref{pointwise} with two lemmas:

\begin{lemma}\label{alltimes1}
For all $t \geq 1$ and $0<\e<1$,
\begin{align*}
\abs{\psi_j(t+\e) - \psi_j (t)}\ll \e.
\end{align*}
\end{lemma}

\begin{proof}
Since $|\phi_s|\le 1$, we have
\begin{align*}
\abs{\psi_j(t+\e) - \psi_j (t)} 
	&= \Abs{\int_{G}\phi_{s_j}(g)\,\left(\frac{\one_{B_{t+\e}}(g)}{m(B_{t+\e})} -
            \frac{\one_{B_t}(g)}{m(B_t)}\right)\,dm(g)} \\
	&\ll \left\|\frac{\one_{B_{t+\e}}(g)}{m(B_{t+\e})} -
            \frac{\one_{B_t}(g)}{m(B_t)}\right\|_1,
\end{align*}
and by (\ref{eq:reg}),
\begin{equation}\label{eq:bb}
\left\|\frac{\one_{B_{t+\e}}(g)}{m(B_{t+\e})} -
            \frac{\one_{B_t}(g)}{m(B_t)}\right\|_1
=\frac{m(B_{t+\e}\backslash B_t)}{m(B_{t+\e})}\ll \e,
\end{equation}
which completes the proof.
\end{proof}

\begin{lemma}\label{alltimes2}
For $f \in \Linf(X)$, $t \geq 1$, $0<\e<1$, and almost every $x\in X$,
\begin{align*}
&\abs{\pi(\b_{t+\e})f(x) - \pi(\b_{t})f(x)} \ll \norm{f}_{\infty}\,\e.
\end{align*}
\end{lemma}

\begin{proof}
Since for almost every $x\in X$, 
\begin{align*}
\abs{\pi(\b_{t+\e})f(x) - \pi(\b_{t})f(x)} &= \Abs{\int_{G}f(g^{-1}x)\,
\left(\frac{\one_{B_{t+\e}}(g)}{m(B_{t+\e})} -
            \frac{\one_{B_t}(g)}{m(B_t)}\right)\,dm(g)}\\
	&\leq \norm{f}_{\infty}\,
\left\|\frac{\one_{B_{t+\e}}(g)}{m(B_{t+\e})} -
            \frac{\one_{B_t}(g)}{m(B_t)}\right\|_1,
\end{align*}
the lemma follows from (\ref{eq:bb}).
\end{proof}


\begin{proof}[Proof of Theorem~\ref{pointwise}]
Let $\delta=\rho-r$ and $\{t_n\}\subset [1,\infty)$ be an increasing sequence such that each interval
$[m,m+1]$, $m\in \mathbb{N}$, is divided in to $\floor{e^{\d m/2}+1}$ sub-intervals of equal length.  Then,
\begin{equation}\label{eq:finite}
	\sum_{n = 1}^{\infty} t^2_n e^{-\d\,t_n} \le \sum_{m=1}^\infty(m+1)^2\floor{e^{\d m/2}+1}e^{-\d m} <\infty.
\end{equation}

We first show that the claim of the theorem holds along times $t_n$. 
For $f \in \Lii(X)$, we set
\[
	C(\cdot,f):= \left(\sum_{n = 1}^{\infty} e^{\d t_n} \Abs{\pi(\b_{t_n})f - \sum_{j=0}^{k}\psi_j (t_n)\, P_jf}^2\right)^{1/2}.
\]
By Theorem~\ref{mean},
\begin{align*}
	\norm{C(\cdot,f)}_2^2 &\leq \sum_{n = 1}^{\infty} e^{\d t_n}\Norm{\pi(\b_{t_n})f -
          \sum_{j=0}^{k}\psi_j (t_n)\, P_jf}_2^2\\
		&\ll  \left(\sum_{n = 1}^{\infty} t_n^2 e^{-\d t_n}\right)\norm{f}_2^2 < \infty,
\end{align*}
Hence, for almost all $x\in X$, we have $C(x,f)<\infty$, and 
\begin{align*}
	\Abs{\pi(\b_{t_n})f(x) - \sum_{j=0}^{k}\psi_j (t_n)\, (P_j f)(x)} \leq C(x,f)\,e^{-\d t_n/2}.
\end{align*}

Now, let $t \ge 1$, and suppose that $t_n \leq t < t_{n+1}$.  Then by Lemmas~\ref{alltimes1}--\ref{alltimes2},
\begin{align*}
	\Abs{\pi(\b_{t})f - \sum_{j=0}^{k}\psi_j (t)\, P_j f} &\leq \Abs{\pi(\b_{t_n})f - \sum_{j=0}^{k}\psi_j (t_n)\, P_j f} \\
		&\indent+ \abs{\pi(\b_t)f - \pi(\b_{t_n})f} \\
&\indent + \sum_{j=0}^{k}\abs{P_j f}\abs{\psi_j (t) - \psi_j (t_n)}\\
		&\ll C(\cdot,f)\,e^{-\d t_n/2} + \floor{e^{\d\floor{t_n}/2}+1}^{-1}\,\norm{f}_\infty \\
&\indent+ \sum_{j=0}^k\abs{P_jf}\, \floor{e^{\d\floor{t_n}/2}+1}^{-1} \\
		&\ll \left(C(\cdot,f) + \norm{f}_\infty+\sum_{j=0}^k\abs{P_jf}\right)\,e^{-\d t/2}.
\end{align*} 
Since 
$$
\norm{C(\cdot,f)}_2 \ll \norm{f}_2 \leq \norm{f}_\infty
$$
and
$$
\norm{P_jf}_2\le \norm{f}_2\le \norm{f}_\infty,
$$
we deduce that for every $f \in \Linf(X)$,
\begin{align}\label{fixedbound}
	\Norm{\sup_{t\geq 1}e^{\d t/2}\Abs{\pi(\b_t)f - \sum_{j=0}^{k}\psi_j (t)\,P_j f}}_2\ll\norm{f}_\infty.
\end{align}
By \cite[Th.~4]{NS97}, for every $f\in L^2(X)$,
$$
\Norm{\sup_{t\geq 1}\Abs{\pi(\b_t)f}}_2 \ll \|f\|_2.
$$
Therefore, for any $f \in \Lii(X)$, we also have an estimate
\begin{align}\label{eq:2}
	&\Norm{\sup_{t\geq 1}\Abs{\pi(\b_t)f - \sum_{j=0}^{k}\psi_j (t)\,P_j f}}_2 \\
\leq &\Norm{\sup_{t\geq 1}\Abs{\pi(\b_t)f}}_2 
	+\Norm{\sup_{t\geq 1}\Abs{\sum_{j=0}^{k}\psi_j (t)\,P_j f}}_2 \ll \|f\|_2.\nonumber
\end{align}

Now we combine estimates (\ref{fixedbound}) and (\ref{eq:2}) using Stein's Complex Interpolation Theorem
\cite[Sec.~V.4]{sw}.  For a measurable function $\t:X\to[1,\infty)$ and complex parameter $z$, we define
a family of operators

\[
	U_z^\t f(x) := e^{z\d \t(x)/2}\left(\pi(\b_{\t(x)})f(x) - \sum_{j=0}^{k}\psi_j (\t(x))\,(P_jf)(x)\right).
\]
It follows from (\ref{fixedbound}) that when $\hbox{Re}(z)= 1$, the operator
$$
U_z^\t :\Linf(X)\to\Lii(X)
$$
is a bounded, and when $\hbox{Re}(z) = 0$, 
the operator 
$$
U_z^\t :\Lii(X)\to\Lii(X)
$$
is bounded, with bounds independent of $\tau$.
Therefore, by the Complex Interpolation Theorem, for every $u \in (0,1)$, 
the operator 
$$
U_u^\t:\Lp(X)\to\Lq(X)
$$
with
\begin{align*}
\frac{1}{p} = \frac{1-u}{2}\quad\hbox{and}\quad \frac{1}{q} = \frac{1-u}{2} + \frac{u}{2}
\end{align*}
is bounded as well with a bound independent of $\tau$.
By taking a supremum over all $\t$, we deduce that for every $u\in (0,1)$,
\begin{align*}
	\Norm{\sup_{t\geq 1}e^{u\d t/2}\Abs{\pi(\b_t)f - \sum_{j=0}^{k}\psi_j (t)\,P_j f}}_2\ll_u \norm{f}_{2/(1-u)}.
\end{align*}
Let
\[
	C_{u}(x,f) := \sup_{t\geq 1}e^{u\d t/2}\Abs{\pi(\b_t)f(x) - \sum_{j=0}^{k}\psi_j (t)\,(P_j f)(x)}.
\]
Then for almost every $x \in X$, we have $C_{u}(x,f)<\infty$, and
\[
	\Abs{\pi(\b_t)f(x) - \sum_{j=0}^{k}\psi_j (t)\,(P_j f)(x)} \leq C_{u}(x,f)\,e^{-u\d t/2},
\]
proving the theorem.
\end{proof}




\end{document}